\documentclass[12pt]{amsart}

\usepackage[utf8]{inputenc}
\usepackage[margin=1in]{geometry}
\usepackage{amsthm}
\usepackage{enumitem}  
\usepackage[OT2,T1]{fontenc}
\usepackage{amscd, amssymb, enumitem, amsmath,mathrsfs, amsfonts, amsaddr}
\usepackage{url}
\usepackage{tikz}
\usepackage{fullpage}
\usepackage{microtype}
\usetikzlibrary{decorations.pathreplacing}
\usepackage[backref=page]{hyperref}
\usepackage[utf8]{inputenc}
\usepackage[main=english,russian]{babel}
\usepackage[pagewise]{lineno}

\DeclareSymbolFont{cyrletters}{OT2}{wncyr}{m}{n}
\DeclareMathSymbol{\Sha}{\mathalpha}{cyrletters}{"58}
\newcommand{\widebar}[1]{\mkern 1.5mu\overline{\mkern-1.5mu#1\mkern-1.5mu}\mkern 1.5mu}

\newtheorem{theorem}{Theorem}[section]
\newtheorem{lemma}[theorem]{Lemma}
\newtheorem{proposition}[theorem]{Proposition}
\newtheorem*{proposition*}{Proposition}

\newtheorem*{questiona*}{Question A}
\newtheorem*{questionb*}{Question B}
\newtheorem*{theorem*}{Theorem}

\theoremstyle{definition}

\newtheorem{example}[theorem]{Example}

\newtheorem{remark}[theorem]{Remark}
\newtheorem{emptyremark}[theorem]{}

\newtheorem*{acknowledgement}{Acknowledgements}

\theoremstyle{remark}

\title{Reduction types of CM curves}
\author{Mentzelos Melistas}
\address{Charles University, Faculty of Mathematics and Physics, Department of
Algebra, Sokolov\-sk\' a 83, 18600 Praha~8, Czech Republic}
\address{University of Twente, Department of Applied Mathematics, Drienerlolaan 5, 7522 NB Enschede, The Netherlands}

\date{\today}

\begin{document}

\maketitle

\begin{abstract}
    We study the reduction properties of low genus curves whose Jacobian has complex multiplication. In the elliptic curve case, we classify the possible Kodaira types of reduction that can occur. Moreover, we investigate the possible Namikawa Ueno types that can occur for genus $2$ curves whose Jacobian has complex multiplication which is defined over the base field. We also produce bounds on the torsion subgroup of abelian varieties with complex multiplication defined over local fields.
\end{abstract}

\section{Introduction}\label{sectionintroduction}

Let $g\geq 1$ be an integer, let $K$ be a CM field, i.e., $K$ is a totally imaginary quadratic extension of a totally real number field, and assume that $K$ has degree $2g$ over $\mathbb{Q}$. Let $R$ be a complete discrete valuation ring with fraction field $L$ of characteristic $0$ and finite residue field. If $A/L$ is an abelian variety, then we will denote by $\text{End}_L(A)$ the ring of endomorphisms of $A/L$ which are defined over $L$. An abelian variety $A/L$ with complex multiplication by $K$ over the field $L$ is an abelian variety $A/L$ of dimension $g$ together with an embedding $\iota : K \hookrightarrow \text{End}_L^0(A):= \mathbb{Q} \otimes \text{End}_L(A)$. Our definition implies that $A/L$ is isotypic (see \cite[Theorem 1.3.1.1]{chaiconradoort}). Note also that we require that $K$ injects into $\text{End}_L^0(A)$ and not just in $\text{End}^0(A):= \mathbb{Q} \otimes \text{End}_{\widebar{\mathbb{Q}}}(A)$. If $K$ injects into $\text{End}^0(A)$, then we will say that $A/L$ has potential complex multiplication by the field $K$. If $C/L$ is a curve, then we will say that $C/L$ has complex multiplication over $L$ (or is a CM curve over $L$) if the Jacobian $Jac(C)/L$ has complex multiplication over the field $L$.

The study of the reduction properties of abelian varieties with complex multiplication is a classical topic with a rich history. Serre and Tate in \cite{st} proved, as a consequence of the N\'eron Ogg Shafarevich Criterion, that every abelian variety with complex multiplication defined over a complete discrete valuation ring with finite residue field has potentially good reduction. More generally, Oort in \cite{oortgoodandstablereductionofabelianvarieties}, proved the same result under the weaker assumption that the residue field is perfect. Lorenzini in \cite{lorenzini1990}, among other results, proved, in the case where the residue field is algebraically closed, that if $C/L$ is a curve with potentially good reduction with simple Jacobian $Jac(C)/L$ that has complex multiplication over $L$, then the degree of the minimal extension over which $Jac(C)/L$ acquires good reduction has at most three prime divisors. In this paper, we study the possible configurations of the special fiber of the minimal proper regular model of elliptic curves with complex multiplication and of genus $2$ curves with complex multiplication.

Our first result is the following.

\begin{theorem}\label{fullendomorphism}
Let $R$ be a complete discrete valuation ring with valuation $v$, fraction field $L$ of characteristic $0$, and algebraically closed residue field $k_L$ of characteristic $p>0$. Let $E/L$ be an elliptic curve with $j$-invariant $j_E$ that has complex multiplication by an imaginary quadratic field $K$ and let $\mathcal{O}_K$ be the ring of integers of $K$. If $K=\mathbb{Q}(i)$ or $\mathbb{Q}(\sqrt{-3})$, then assume that End$_L(E) \cong \mathcal{O}_K$. Then, depending on $p$, $v(p)$, $j_E$, and $K$, the possible reduction types of $E/L$ are given by the following table \\

    \centering
    \begin{tabular}{|p{1cm}|p{1cm}|p{2cm}|p{4.5cm}|}
    \hline
        $p$ & $v(p)$ & $j_E$ &Possible reduction types \\ \hline 
        $\neq 2$ & any & $\neq 0, 1728$ &  I$_0$ or I$_0^*$ \\ \hline
        $ 2$ & $1$ & $\neq 0, 1728$ & I$_0$, I$_4^*$, I$_8^*$, II, or II$^*$ \\ \hline
        $\neq 2$ & any & $ 1728$ &   I$_0$, III, III$^*$, or I$_0^*$ \\ \hline
        $\neq 3$ & any & $ 0$ &  I$_0$, II, II$^*$, IV, IV$^*$ or I$_0^*$ \\ \hline
    \end{tabular}
\end{theorem}

Keeping the same notation as above, if we do not assume that the complex multiplication is defined over $L$, then, as Theorem \ref{potentialcm} below shows, in addition to the Kodaira types of Theorem \ref{fullendomorphism}, a few more Kodaira types can also occur.

\begin{theorem}\label{potentialcm} 
Let $R$ be a complete discrete valuation ring with fraction field $L$ of characteristic $0$ and algebraically closed residue field $k_L$ of characteristic $p>0$. Let $E/L$ be an elliptic curve with potential complex multiplication by an imaginary quadratic field $K$ and denote by $j_E$ the $j$-invariant of $E/L$. If $K=\mathbb{Q}(i)$ or $\mathbb{Q}(\sqrt{-3})$, then assume that End$_{LK}(E) \cong \mathcal{O}_K$. Then, depending on $p$, $v(p)$, $j_E$, and $K$, the possible reduction types of $E/L$ are given by the following table \\

\centering
    \begin{tabular}{|p{1cm}|p{1cm}|p{2cm}|p{4.5cm}|}
    \hline
        $p$ & $v(p)$ & $j_E$ & Possible reduction types \\ \hline
        $\neq 2$ & any & $\neq 0, 1728$ &  I$_0$, III, III$^*$, or I$_0^*$. \\ \hline
        $\neq 2$ & any & $ 1728$ &   I$_0$, III, III$^*$, or I$_0^*$ \\ \hline
         $2$ & $1$ & $ 1728$ &  I$_0$, II, III, III$^*$, I$_2^*$ or I$_3^*$ \\ \hline
        $\neq 3$ & any & $ 0$ &   I$_0$, II, II$^*$, IV, IV$^*$ or I$_0^*$ \\ \hline
 
    \end{tabular}
\end{theorem}

 In Section \ref{sectioncmellipticcurves} below, we present examples showing that all the reduction types of Theorem \ref{fullendomorphism} and of Theorem \ref{potentialcm} do indeed occur.

Curves of genus $2$ with complex multiplication have received a lot of interest lately, especially due to their cryptographic applications. If an elliptic curve has complex multiplication, then it has potentially good reduction. However, it is not true that every curve of genus $2$ whose Jacobian has complex multiplication has potentially good reduction. For a primitive CM field Goren and Lauter, in \cite{gorenlauter2007}, proved a bound on the primes of geometric bad reduction for curves of genus $2$ whose Jacobian has complex multiplication.

Let $R$ be a complete discrete valuation ring with fraction field $L$ of characteristic $0$ and algebraically closed residue field of characteristic $p>0$. Let $C/L$ be a projective, smooth, and geometrically connected curve of genus $2$ and let $C^{min}/R$ be its minimal proper regular model. There exists a complete classification of the possibilities for the special fiber of $C^{min}/R$ (see \cite{namikawauenoclassification}). In this classification, there are more than 120 possibilities, referred to as reduction types. Moreover, Liu in \cite{Liugenus2algorithm} has produced an algorithm that computes the special fiber of $C^{min}/R$ under the assumption that the extension of minimal degree over which $C/L$ acquires stable reduction is tame, i.e., the degree of this extension is not divisible by $p$. We note that this tameness assumption is automatically satisfied if $p > 5$.

Throughout this article, if $A/L$ is any variety over a field $L$ and $M/L$ is a field extension, then we will denote by $A_M/M$ the base change of $A/L$ to $M$. Among other results in Section \ref{sectiongenus2cm}, we prove the following theorem (for the reduction types we follow Liu's notation in \cite{Liugenus2algorithm}).

\begin{theorem}\label{genus2jacobiangcurvegoodreduction}
Let $R$ be a complete discrete valuation ring with fraction field $L$ of characteristic $0$ and finite residue field $k_L$ of characteristic $p>5$. Let $C/L$ be a projective, smooth, and geometrically connected curve of genus $2$ with simple Jacobian $Jac(C)/L$ that has complex multiplication by a quartic CM field $K$ over the field $L$.  Let $\mu'=|\mu(K)|$, where $\mu(K)$ is the group of roots of unity in $K$, and let $L^{unr}$ be the maximal unramified extension of $L$. 

\begin{enumerate}
    \item Assume that $C/L$ has potentially good reduction. Then the possible special fibers of the minimal proper regular model of $C_{L^{unr}}/L^{unr}$ are given by the following table. 
\begin{table}[!ht]
    \centering
    \begin{tabular}{|p{1.5cm}|p{11.5cm}|}
    \hline
        $\mu'$ & Possible reduction types of $C_{L^{unr}}/L^{unr}$ \\ \hline
        $2$ &   $[I_{0-0-0}]$, $[I_{0-0-0}^*]$ \\ \hline
        $4$ & $[I_{0-0-0}]$, $[I_{0-0-0}^*]$, $[VI]$ \\ \hline
        $6$ & $[I_{0-0-0}]$, $[I_{0-0-0}^*]$, $[III]$, $[IV]$ \\ \hline
        $8$ & $[I_{0-0-0}]$, $[I_{0-0-0}^*]$, $[VI]$, $[VII]$, $[VII^*]$ \\ \hline
        $10$ &   $[I_{0-0-0}]$, $[I_{0-0-0}^*]$, $[IX-1]$, $[IX-2]$, $[IX-3]$, $[IX-4]$, $[VIII-1]$, $[VIII-2]$,$[VIII-3]$, $[VIII-4]$ \\ \hline
        $12$ & $[I_{0-0-0}]$, $[I_{0-0-0}^*]$, $[III]$, $[IV]$, $[VI]$ \\ \hline
    \end{tabular}
\end{table}

   \item Assume that $C/L$ does not have potentially good reduction. Then the possible special fibers of the minimal proper regular model of $C_{L^{unr}}/L^{unr}$ are given by the following table, where $d$ and $r$ are defined as in \cite[Section 4.3]{Liugenus2algorithm}.
\begin{table}[!ht]
    \centering
    \begin{tabular}{|p{1.5cm}| p{11.5cm}|}
    \hline
        $\mu'$ &  Possible reduction types of $C_{L^{unr}}/L^{unr}$ \\ \hline
        $2$ or $10$ &   $[I_0^*-I_0^*-(d-2)/2]$ \\ \hline
        $4$  & $[I_0^*-I_0^*-(d-2)/2]$, $[III-III-(d-2)/4]$, $[III-III^*-(d-4)/4]$, $[III^*-III^*-(d-6)/4]$, $[2I_0^*-(r-1)/2]$ \\ \hline
        $6$ &  $[I_0^*-I_0^*-(d-2)/2]$, $[IV-IV-(d-2)/3]$, $[IV-IV^*-(d-3)/3]$, $[IV^*-IV^*-(d-4)/3]$, $[II-II-(d-2)/6]$, $[II-II^*-(d-6)/6]$, $[II^*-II^*-(d-10)/6]$, $[I_0^*-II-(d-4)/6]$,
    $[I_0^*-II^*-(d-8)/6]$, $[2IV-(r-1)/3$, $[2IV^*-(r-2)/3]$ \\ \hline
        $8$  & $[I_0^*-I_0^*-(d-2)/2]$, $[III-III-(d-2)/4]$, $[III-III^*-(d-4)/4]$, $[III^*-III^*-(d-6)/4]$, $[2I_0^*-(r-1)/2]$, $[2III-(r-1)/4]$, $[2III^*-(r-3)/4]$ \\ \hline
        $12$ & $[I_0^*-I_0^*-(d-2)/2]$, $[III-III-(d-2)/4]$, $[III-III^*-(d-4)/4]$, $[III^*-III^*-(d-6)/4]$, $[2I_0^*-(r-1)/2]$ $[IV-IV-(d-2)/3]$, $[IV-IV^*-(d-3)/3]$, $[IV^*-IV^*-(d-4)/3]$, $[II-II-(d-2)/6]$, $[II-II^*-(d-6)/6]$, $[II^*-II^*-(d-10)/6]$, $[I_0^*-II-(d-4)/6]$,
        $[I_0^*-II^*-(d-8)/6]$,  $[2IV-(r-1)/3$, $[2IV^*-(r-2)/3]$, $[2II-(r-1)/6]$, $[2II^*-(r-5)/6]$ \\ \hline
    \end{tabular}
\end{table}
\end{enumerate}
\end{theorem}

When $\mu'=8$ or $10$, assuming that $C/L$ has potentially good reduction and the special fiber of the stable model is not isomorphic to either the curve $C_0$ or the curve $C_1$ of \ref{c0c1}, we obtain, in Section \ref{sectiongenus2cm}, a more precise list of possible reduction types (see Theorem \ref{theorem4.5} below).

In the last section, we focus on component groups and torsion points of CM abelian varieties. Using ideas of Clark and Xarles from \cite{cx} we prove the following proposition.

\begin{proposition}\label{propositioncmandtorsion}
Let $R$ be a complete discrete valuation ring with fraction field $L$ of characteristic $0$ and finite residue field $k_L$ which has characteristic $p$ and cardinality $q$. Denote by $e$ the absolute ramification index of $L$. Let $A/L$ be an abelian variety with complex multiplication by a CM field $K$ over $L$. Then 
$$|A(L)_{tors}| \leq \text{max} \{ |\mu(K)| \cdot p^{2g \gamma_p(e|\mu(K)|)} ,  \lfloor (1+\sqrt{q})^2 \rfloor^g  \cdot p^{2g \gamma_p(e)} \},$$ where $\gamma_p(m)=\lfloor \log_p(\frac{pm}{p-1}) \rfloor$.
\end{proposition}

This article is organized as follows. In section \ref{sectioncmellipticcurves} we consider the elliptic curve case and we prove Theorem \ref{fullendomorphism} and
Theorem \ref{potentialcm}. Section \ref{cmabelianvarieties} mostly contains background material on reduction of abelian varieties used in the last two sections. After briefly recalling some basic background on reduction of genus $2$ curves, we prove Theorem \ref{genus2jacobiangcurvegoodreduction} in Section \ref{sectiongenus2cm}. Finally, in Section \ref{sectioncomponentgroups} we study the  possible geometric component groups of abelian varieties and we prove Proposition \ref{propositioncmandtorsion}.

\begin{acknowledgement}
The author would like to thank Pete L. Clark for suggesting the study of the properties of bad reduction of elliptic curves with complex multiplication. I would like to thank Dino J. Lorenzini for some very helpful email correspondence and some insightful comments on an earlier version of this manuscript. I would also like to thank an anonymous referee for many valuable comments and suggestions that improved my manuscript. The author was supported by Czech Science Foundation (GA\v CR) grant 21-00420M and by Charles University Research Center program No.UNCE/SCI/022.
\end{acknowledgement}

\section{Kodaira types of CM elliptic curves}\label{sectioncmellipticcurves}

In this section we prove Theorems \ref{fullendomorphism} and \ref{potentialcm}, and we also present a few examples. We first begin with the proof of Theorem \ref{fullendomorphism} as it will be needed in the proof of Theorem \ref{potentialcm}.

Our starting point for our proofs is Lemma \ref{twistlemma} below, whose proof can be found in \cite[Lemma 2.4]{conradgzrevisited}.  We note that when $E/L$ is an elliptic curve defined over a number field, then Lemma \ref{twistlemma} was originally due to Serre and Tate \cite[Page 507]{st} (see also \cite[Corollary 5.22]{rubincmnotes}).

\begin{lemma}\label{twistlemma}
Let $R$ be a complete discrete valuation ring with fraction field $L$ of characteristic $0$ and algebraically closed residue field $k_L$. Let $E/L$ be an elliptic curve with complex multiplication by an imaginary quadratic field $K \subset L$. If $K=\mathbb{Q}(i)$ or $\mathbb{Q}(\sqrt{-3})$, then assume that End$_L(E) \cong \mathcal{O}_K$, where $\mathcal{O}_K$ is the ring of integers of $K$. Then there exists an elliptic curve $E'/L$ such that $E'/L$ has good reduction and the curves $E/L$ and $E'/L$ become isomorphic over an algebraic closure of $L$.
\end{lemma}

We also record here the following useful lemma, which gives the reduction type of a quadratic twist of an elliptic curve with good reduction.

\begin{lemma}\label{lemmatwists}
    Let $R$ be a complete discrete valuation ring with valuation $v$, fraction field $L$ of characteristic $0$, and algebraically closed residue field $k_L$ of characteristic $p>0$. Let $E/L$ be an elliptic curve good reduction and let $E'/L$ be a quadratic twist of $E/L$.
    \begin{enumerate}
        \item If $p>2$, then $E'/L$ has either good reduction or reduction of type I$_0^*$.
        \item If $p=2$ and $v(2)=1$, then $E'/L$ has either good reduction, or reduction of type I$_8^*$, I$_4^*$, II, or II$^*$.
    \end{enumerate}
\end{lemma}
\begin{proof}
    Part $(i)$ is well known (see e.g. \cite[Proposition 1]{com}). Part $(ii)$ follows by combining explicit formulas for quadratic twists of elliptic curves (see \cite[Proposition 5.7.1]{ellipticcurvehandbook}) along with \cite[Tableau IV]{pap}. Alternatively, one can use \cite[Theorem 4.2]{lorenzini2013} together with \cite{haiyang}.
\end{proof}

\begin{proof}[Proof of Theorem \ref{fullendomorphism}]
Assume that $p\neq 2$ and that $j_E \neq 0, 1728$. Lemma \ref{twistlemma} tells us that there exists an elliptic curve $E'/L$ with good reduction which becomes isomorphic to $E/L$ over an algebraic closure of $L$. Since $j_E \neq 0, 1728$, we see that $E'/L$ is a quadratic twist of $E/L$ (see \cite[Section X.5]{aec}). Therefore, since $p \neq 2$, using Part $(i)$ of Lemma \ref{lemmatwists}, we find that $E/L$ has good reduction or reduction of type I$_0^*$. 

Assume that $p = 2$, $v(2)=1$, and that $j_E \neq 0, 1728$. Proceeding in the same way as in the previous paragraph we see that there exists a quadratic twist $E'/L$ of $E/L$ which has good reduction. Since $v(2)=1$, by Part $(ii)$ of Lemma \ref{lemmatwists} we obtain that $E/L$ has good reduction or reduction of type I$_8^*$, I$_4^*$, II, or II$^*$. 

Assume that $p \neq 2$ and that $j_E=1728$. Using \cite[Proposition III.1.4]{aec} and \cite[Proposition X.5.4]{aec} we find that the elliptic curve $E/L$ has a short Weierstrass equation of the form $$y^2=x^3+Ax,$$ for some $A \in L^*$. The discriminant of this Weierstrass equation is $\Delta=-64A^3$. Since $p \neq 2$ by assumption, we have that $v(\Delta)=3v(A)$ and, hence, $3$ divides  $v(\Delta)$. Let $\Delta_{min}$ be the discriminant of a minimal Weierstrass equation for $E/L$. Since when a change of Weierstrass equation is performed the valuation of the discriminant changes by a multiple of $12$, we find that $v(\Delta)-v(\Delta_{min})$ is a multiple of $12$. Since $3$ divides $v(\Delta)$, we obtain that $3$ divides $v(\Delta_{min})$. 
 
 Case 1: $p\geq 5$. Using Tate's algorithm \cite{tatealgorithm} (see also \cite[Page 365]{silverman2}) and using the fact that $E/L$ has potentially good reduction, we find that $E/L$ can only have good reduction or reduction of type I$_0^*$, III, or III$^*$. Case 2: $p=3$.  Choose a minimal Weierstrass equation for $E/L$ and let $c_4$, $c_6$, and $\Delta_{min}$ be the $c_4$-invariant, the $c_6$-invariant, and the discriminant of this equation, respectively. Since $j_E=1728$, $j_E=\frac{c_4^3}{\Delta_{min}}$, and $1728\Delta_{min}=c_4^3-c_6^2$ (see \cite[Page 42]{aec}), we find that $c_6=0$. Therefore, using \cite[Tableau III]{pap} we find that $E/L$ can only have good reduction or reduction of type I$_0^*$, III, or III$^*$.

 Assume that $j_E=0$. Using \cite[Proposition III.1.4]{aec} and \cite[Proposition X.5.4]{aec} we find that the curve $E/L$ has a short Weierstrass equation of the form $$y^2=x^3+B,$$ for some $B \in L^*$. The discriminant of this Weierstrass equation is $\Delta= -432B^2$.  Since $p \neq 3$ by assumption, we have that $v(\Delta)=4v(2)+2v(B)$ and, hence, $2$ divides  $v(\Delta)$.  Let $\Delta_{min}$ be the discriminant of a minimal Weierstrass equation for $E/L$. Since when a change of Weierstrass equation is performed the valuation of the discriminant changes by a multiple of $12$, we find that $v(\Delta)-v(\Delta_{min})$ is a multiple of $12$. Since $2$ divides $v(\Delta)$, we obtain that $2$ divides $v(\Delta_{min})$. 

Case 1: $p \geq 5$. Using Tate's algorithm \cite{tatealgorithm}, we obtain that $E/L$ cannot have reduction type III or III$^*$. Therefore, it can only have good reduction or reduction of type II, II$^*$, IV, IV$^*$, or I$_0^*$. Case 2: $p =2$. Choose a minimal Weierstrass equation for $E/L$ and let $c_4$ and $\Delta_{min}$ be the $c_4$-invariant and the discriminant of this equation, respectively. Since $j_E=0$ and $j_E=\frac{c_4^3}{\Delta_{min}}$, we find that $c_4=0$. Therefore, using \cite[Tableau V]{pap} we find that $E/L$ can only have good reduction or reduction of type II, II$^*$, IV, IV$^*$, or I$_0^*$. This completes the proof of our theorem.

\end{proof}

We now present examples showing that all the possible reduction types of Theorem \ref{fullendomorphism} do indeed occur.

\begin{example}
Consider the elliptic curve $E/\mathbb{Q}(\sqrt{-11})$ given by the following Weierstrass equation $$y^2+y=x^3+ax+(a-3)x-2,$$ where $a=\frac{1+ \sqrt{-11}}{2}$. This is the curve with LMFDB \cite{lmfdb} label \href{https://www.lmfdb.org/EllipticCurve/2.0.11.1/9.1/CMa/1}{2.0.11.1-9.1-CMa1}, which has complex multiplication over $\mathbb{Q}(\sqrt{-11})$, $j$-invariant equal to $-32768$, and only one prime of additive reduction with Kodaira type I$_0^*$.
\end{example}

\begin{example}\label{exampleiiifullendomorphism1}
Consider the curve $E_1/\mathbb{Q}(\sqrt{-7})$ given by the following Weierstrass equation $$y^2+axy=x^3+(-a-1)x+1,$$ where $a=\frac{1+\sqrt{-7}}{2}$. Note that the prime $(a)$ of $\mathbb{Q}(\sqrt{-7})$ lies above $(2)$.  The curve $E_1/\mathbb{Q}(\sqrt{-7})$ has LMFDB \cite{lmfdb} label \href{https://www.lmfdb.org/EllipticCurve/2.0.7.1/16.1/CMa/1}{2.0.7.1-16.1-CMa1}, $j$-invariant equal to $-3375$, and additive reduction of Kodaira type I$_4^*$ at $(a)$.

Consider the curve $E_2/\mathbb{Q}(\sqrt{-7})$ given by the following Weierstrass equation $$y^2+axy+ay=x^3+(-a-1)x^2+(2a+2)x-2a+3.$$  The curve $E_2/\mathbb{Q}(\sqrt{-7})$ has LMFDB \cite{lmfdb} label \href{https://www.lmfdb.org/EllipticCurve/2.0.7.1/64.1/CMa/1}{2.0.7.1-64.1-CMa1}, $j$-invariant equal to $-3375$, and additive reduction of Kodaira type I$_8^*$ at $(a)$.

\end{example}

\begin{example}\label{exampleiiifullendomorphism2}
Consider the curve $E_1/\mathbb{Q}(\sqrt{-11})$ given by the following Weierstrass equation $$y^2=x^3+(a+1)x^2+(a+2)x+1,$$ where $a=\frac{1+\sqrt{-11}}{2}$.  The curve $E_1/\mathbb{Q}(\sqrt{-11})$ has LMFDB \cite{lmfdb} label \href{https://www.lmfdb.org/EllipticCurve/2.0.11.1/4096.1/CMb/1}{2.0.11.1-4096.1-CMb1}, $j$-invariant equal to $-32768$, and additive reduction of Kodaira type II at $(2)$.

Consider the curve $E_2/\mathbb{Q}(\sqrt{-11})$ given by the following Weierstrass equation $$y^2=x^3+(a+1)x^2+(a+10)x+12a-1.$$  The curve $E_2/\mathbb{Q}(\sqrt{-11})$ has LMFDB \cite{lmfdb} label \href{https://www.lmfdb.org/EllipticCurve/2.0.11.1/256.1/CMb/1}{2.0.11.1-256.1-CMb1}, $j$-invariant equal to $-32768$, and additive reduction of Kodaira type II$^*$ at $(2)$.
\end{example}

\begin{example}\label{examplesjequalto1728}
    Consider the elliptic curves $E_1/\mathbb{Q}(i)$, $E_2/\mathbb{Q}(i)$, $E_3/\mathbb{Q}(i)$ given by LMFDB \cite{lmfdb} labels \href{https://www.lmfdb.org/EllipticCurve/2.0.4.1/2025.1/CMa/1}{2.0.4.1-2025.1-CMa1}, \href{https://www.lmfdb.org/EllipticCurve/2.0.4.1/2025.1/CMb/1}{2.0.4.1-2025.1-CMb1}, and	\href{https://www.lmfdb.org/EllipticCurve/2.0.4.1/2025.1/CMc/1}{2.0.4.1-2025.1-CMc1}, respectively. All these elliptic curves have complex multiplication by $\mathbb{Q}(i)$ (and $j$-invariant equal to $1728$). Moreover, they have bad reduction at $3$ of Kodaira type III, I$_0^*$, and III$^*$, respectively. 
\end{example}

\begin{example}\label{examplesjequalto0}
   Consider the elliptic curves $E_1/\mathbb{Q}(\sqrt{-3})$, $E_2/\mathbb{Q}(\sqrt{-3})$, $E_3/\mathbb{Q}(\sqrt{-3})$, $E_4/\mathbb{Q}(\sqrt{-3})$ given by LMFDB \cite{lmfdb} labels \href{https://www.lmfdb.org/EllipticCurve/2.0.3.1/256.1/CMa/1}{2.0.3.1-256.1-CMa1}, \href{https://www.lmfdb.org/EllipticCurve/2.0.3.1/784.1/CMa/1}{2.0.3.1-784.1-CMa1}, \href{https://www.lmfdb.org/EllipticCurve/2.0.3.1/784.3/CMb/1}{2.0.3.1-784.3-CMb1}, and \href{https://www.lmfdb.org/EllipticCurve/2.0.3.1/4096.1/CMb/1}{2.0.3.1-4096.1-CMb1}, respectively. All these elliptic curves have complex multiplication by $\mathbb{Q}(\sqrt{-3})$ (and $j$-invariant equal to $0$). Moreover, they have bad reduction at $2$ of Kodaira type II, IV, IV$^*$, and I$_0^*$ respectively. 
   
   We can also find an elliptic curve $E_5/\mathbb{Q}(\sqrt{-3})$ complex multiplication by $\mathbb{Q}(\sqrt{-3})$ having bad reduction at $2$ of Kodaira type II$^*$ as follows. Start with an elliptic curve $E/\mathbb{Q}$ that has potential complex multiplication by $\mathbb{Q}(\sqrt{-3})$ and bad reduction at $2$ of Kodaira type II$^*$ (for example the curve with LMFDB label \href{https://www.lmfdb.org/EllipticCurve/Q/1728/m/1}{1728.m1}). Then consider the base change of $E/\mathbb{Q}$ to $\mathbb{Q}(\sqrt{-3})$ which has complex multiplication by $\mathbb{Q}(\sqrt{-3})$. Since $2$ is unramified in $\mathbb{Q}(\sqrt{-3})$, the base change will still have bad reduction at $2$ of Kodaira type II$^*$.
\end{example}

\begin{example}
     Consider the elliptic curve $E/\mathbb{Q}(\sqrt{-3})$ given by the following Weierstrass equation $$y^2+y=x^3-30x+63.$$ This is the curve with LMFDB \cite{lmfdb} label \href{https://www.lmfdb.org/EllipticCurve/2.0.3.1/81.1/CMa/2}{2.0.3.1-81.1-CMa2} and we have that End$_L(E) \cong \mathbb{Z}[(1+\sqrt{-27})/2]$. Moreover, $E/\mathbb{Q}(\sqrt{-3})$ has only one prime of additive reduction with Kodaira type IV$^*$ and $j$-invariant equal to $-12288000$. We note that in this example $L=K=\mathbb{Q}(\sqrt{-3})$.
\end{example}

We now proceed to the proof of Theorem \ref{potentialcm}.

\begin{proof}[Proof of Theorem \ref{potentialcm}] 

Assume that $p\neq 2$ and that $j_E \neq 0, 1728$. Let $\mathcal{O}_K$ be the ring of integers of $K$, let $F$ be the compositum $LK$, and let $E_F/F$ be the base change of $E/L$ to $F$. Since $K \subset F$, the curve $E_F/F$ has complex multiplication and, hence, it follows from Theorem \ref{fullendomorphism} that $E_{F}/F$ has good reducion or reduction of type I$_0^*$.
Since the ramification index $e(F/L)$ is either $1$ or $2$, the extension $F/L$ is tame because we assume that $p \neq 2$. Therefore, since $E_{F}/F$ has either good reduction or reduction of type I$_0^*$, using \cite[Theorem 3]{dokchitsertameextension}, we find that $E/L$ can only have good reduction or reduction of type III, III$^*$, or I$_0^*$. 

Assume now that $j_E=1728$ and that $v(2)=1$. Proceeding in the same way as in the proof of Theorem \ref{fullendomorphism} we see that $E/L$ has a short Weierstrass equation of the form $$y^2=x^3+Ax,$$ for some $A \in L^*$. The discriminant and $c_6$ invariant of this Weierstrass equation are $\Delta=-64A^3$ and $0$. Since $v(2)=1$ by assumption, we have that $v(\Delta)=6+3v(A)$ and, hence, $3$ divides  $v(\Delta)$. Let $\Delta_{min}$ be the discriminant of a minimal Weierstrass equation for $E/L$. Since when a change of Weierstrass equation is performed the valuation of the discriminant changes by a multiple of $12$, we find that $v(\Delta)-v(\Delta_{min})$ is a multiple of $12$. Since $3$ divides $v(\Delta)$, we obtain that $3$ divides $v(\Delta_{min})$. Therefore, since $v(2)=1$, using \cite[Tableau V]{pap} we find that $E/L$ can only have good reduction or reduction of type I$_0$, II, III, III$^*$, I$_2^*$ or I$_3^*$. Finally, the proof of other two cases is exactly the same as in Theorem \ref{fullendomorphism}.
\end{proof}

\begin{example}
Consider the elliptic curve $E_1/\mathbb{Q}$ given by the following Weierstrass equation $$y^2+xy=x^3-x^2-2x-1.$$ This is the curve with LMFDB \cite{lmfdb} label \href{https://www.lmfdb.org/EllipticCurve/Q/49/a/4}{49.a4} and $j$-invariant equal to $-3375$. The curve $E_1/\mathbb{Q}$ has potential complex multiplication by $K=\mathbb{Q}(\sqrt{-7})$ and it has bad reduction modulo $7$ of Kodaira type III.

Consider now the elliptic curve $E_2/\mathbb{Q}$ given by the following Weierstrass equation $$y^2+xy=x^3-x^2-1822x+30393.$$ This is the curve with LMFDB \cite{lmfdb} label \href{https://www.lmfdb.org/EllipticCurve/Q/49/a/1}{49.a1} and $j$-invariant equal to $16581375$. The curve $E_2/\mathbb{Q}$ has potential complex multiplication by $K=\mathbb{Q}(\sqrt{-7})$ and it has bad reduction modulo $7$ of Kodaira type III$^*$.

Thus, both the additional reduction types of Theorem \ref{potentialcm} do occur. We note that our examples are isogenous, so both reduction types can also occur in the same isogeny class.
\end{example}

\begin{example}
    Consider the elliptic curves $E_1/\mathbb{Q}$, $E_2/\mathbb{Q}$, $E_3/\mathbb{Q}$, $E_4/\mathbb{Q}$, and $E_5/\mathbb{Q}$ given by LMFDB \cite{lmfdb} labels \href{https://www.lmfdb.org/EllipticCurve/Q/32/a/3}{32.a3}, \href{https://www.lmfdb.org/EllipticCurve/Q/32/a/4}{32.a4}, \href{https://www.lmfdb.org/EllipticCurve/Q/64/a/3}{64.a3}, \href{https://www.lmfdb.org/EllipticCurve/Q/64/a/4}{64a4}, \href{https://www.lmfdb.org/EllipticCurve/Q/256/b/2}{256.b2}, respectively. All these elliptic curves have potential complex multiplication by $\mathbb{Q}(i)$ (and $j$-invariant equal to $1728$).  Moreover, they have reduction at $2$ of type III, I$_3^*$, I$_2^*$, II, and III$^*$, respectively. 
\end{example}

\begin{proposition}
    Let $R$ be a complete discrete valuation ring with valuation $v$, fraction field $L$ of characteristic $0$, and algebraically closed residue field $k_L$ of characteristic $p$. Let $E/L$ be an elliptic curve with $j$-invariant $j_E$.
    \begin{enumerate}
        \item  If $p=2$ and $j_E=1728$, then $E/L$ can not have reduction of type IV or IV$^*$.
        \item If $p=3$, $j_E=0$, and $v(3)$ is even, then $E/L$ can not have reduction of type III or III$^*$.
    \end{enumerate}

\end{proposition}
\begin{proof}
  
 Assume that $p=2$ and that $j_E=1728$. Proceeding in the same way as in the proof of Theorem \ref{fullendomorphism} we see that $E/L$ has a short Weierstrass equation of the form $$y^2=x^3+Ax,$$ for some $A \in L^*$. The discriminant of this Weierstrass equation is $\Delta=-64A^3$. We have that $v(\Delta)=6v(2)+3v(A)$ and, hence, $3$ divides  $v(\Delta)$. Let $\Delta_{min}$ be the discriminant of a minimal Weierstrass equation for $E/L$. Since when a change of Weierstrass equation is performed the valuation of the discriminant changes by a multiple of $12$, we find that $v(\Delta)-v(\Delta_{min})$ is a multiple of $12$. Since $3$ divides $v(\Delta)$, we obtain that $3$ divides $v(\Delta_{min})$. Therefore, using \cite[Tableau V]{pap} we find that $E/L$ cannot have reduction of type IV or IV$^*$.

  Assume that $p=3$, $j_E=0$, and $v(3)$ is even.  Proceeding in the same way as in the proof of Theorem \ref{fullendomorphism} we see that $E/L$ has a short Weierstrass equation of the form $$y^2=x^3+B,$$ for some $B \in L^*$. The discriminant of this Weierstrass equation is $\Delta= -432B^2$.  Since $p = 3$ is even by assumption, we have that $v(\Delta)=3v(3)+2v(B)$ and, hence, $2$ divides  $v(\Delta)$ because $v(3)$ is even.  Let $\Delta_{min}$ be the discriminant of a minimal Weierstrass equation for $E/L$. Since when a change of Weierstrass equation is performed the valuation of the discriminant changes by a multiple of $12$, we find that $v(\Delta)-v(\Delta_{min})$ is a multiple of $12$. Since $2$ divides $v(\Delta)$, we obtain that $2$ divides $v(\Delta_{min})$. Therefore, using \cite[Tableau III]{pap} we find that $E/L$ cannot have reduction of type III or III$^*$.
\end{proof}

\section{Abelian varieties with complex multiplication}\label{cmabelianvarieties}

 In this section we first prove a general lemma which is a consequence of the N\'eron Ogg Shafarevich Criterion and is essentially due to Serre and Tate \cite{st}. Then we recall a few basic facts concerning reduction of abelian varieties and reduction of genus $2$ curves.

\begin{lemma}\label{lemmarootsofunity}
Let $R$ be a complete discrete valuation ring with fraction field $L$ of characteristic $0$ and finite residue field $k_L$ of characteristic $p>0$. Let $A/L$ be an abelian variety with complex multiplication by $K$ over $L$ and let $\mu'= | \mu (K) | $, where $\mu (K)$ are the roots of unity contained in $K$. Then there exists a finite extension $M/L^{unr}$, where $L^{unr}$ is the maximal unramified extension of $L$, which has degree dividing $\mu'$ and such that the base change $A_M/M$ has good reduction.
\end{lemma}
\begin{proof}
This is a consequence of results of Serre and Tate in \cite{st}. We include some of the details for completeness. Fix a separable closure $\widebar{L}$ of $L$. Let $\ell \neq p $ be a prime and let $\rho_{\ell}: \text{Gal}(\widebar{L}/L) \longrightarrow \text{Aut}(T_{\ell}(A))$ be the $\ell$-adic Galois representation of $A/L$. Let $v$ be the valuation of $L$, let $\bar{v}$ be the extension of $v$ to $\widebar{L}$, and let $I(\bar{v})$ be the inertia group of $\bar{v}$. Note that the extension of $v$ to $\widebar{L}$ is unique because $R$ is complete. Since $A/L$ has complex multiplication by $K$ over $L$, the image $\rho_{\ell}(I(\bar{v}))$  is contained in $\mu (K)$ (see \cite[Theorem 6]{st}). Let now $M/L^{unr}$ be the minimal Galois extension over which $A_{L^{unr}}/L^{unr}$ acquires good reduction. Such an extension exists by \cite[Theorem 6]{st} and \cite[Corollary 3]{st}. Then, we have that $\text{ker}({\rho_{\ell}}\vert_{I(\bar{v})})=\text{Gal}(\widebar{L}/M)$ (by \cite[Corollary 3]{st}) and that $|\text{Gal}(M/L^{unr})|=|\rho_{\ell}(I(\bar{v}))|$ which divides $\mu'$. This proves our lemma.
\end{proof}

\begin{example}(see also \cite[Example 3.1]{mentzeloskodairaandtorsion})
Let $p$ be an odd prime and $s$ be an integer with $1 \leq s \leq p-2$. Consider the smooth projective curve $C_{p,s}/\mathbb{Q}$ birational to $$y^p=x^s(1-x).$$
The curve $C_{p,s}/\mathbb{Q}$ has genus $\frac{p-1}{2}$. The Jacobian $J_{p,s}/\mathbb{Q}$ of $C_{p,s}/\mathbb{Q}$ has complex multiplication by $K=\mathbb{Q}(\zeta_p)$ defined over $\mathbb{Q}(\zeta_p)$ (see \cite[Page 202]{someresultsgrossrohrlich}).
It turns out that $J_{p,s}/\mathbb{Q}$ has good reduction away from $p$ and potentially good reduction modulo $p$. When  $C_{p,s}/\mathbb{Q}$ is tame (see \cite[Example 5.1]{lorenzini1993} for the definition) $J_{p,s}/\mathbb{Q}$ has purely additive reduction modulo $p$ and achieves good reduction after a totally ramified extension of degree $2(p-1)$ (see Anmerkung in \cite[Page 339]{maedastablereductionfermatcurve} for the last statement).
\end{example}

Let $R$ be a complete discrete valuation ring with fraction field $L$ of characteristic $0$ and perfect residue field $k_L$. Let $A/L$ be an abelian variety of dimension $g$. We denote by $\mathcal{A}/R$ the N\'eron model of $A/L$ (see \cite{neronmodelsbook} for the definition as well as the basic properties of N\'eron models). The special fiber $\mathcal{A}_{k_L}/k_L$ of $\mathcal{A}/R$ is a smooth commutative group scheme. We denote by $\mathcal{A}^0_{k_L}/k_L$ the connected component of the identity of $\mathcal{A}_{k_L}/k_L$. Since $k_L$ is perfect, by a theorem of Chevalley (see \cite[Theorem 1.1]{con}) we have a short exact sequence $$0\longrightarrow T \times U \longrightarrow \mathcal{A}^0_{k_L} \longrightarrow B \longrightarrow 0, $$
where $T/k_L$ is a torus, $U/k_L$ is a unipotent group, and $B/k_L$ is an abelian variety. The number $\text{dim}(U)$ (resp. $\text{dim}(T)$, $\text{dim}(B)$) is called the unipotent (resp. toric, abelian) rank of $A/L$. By construction, $g=\text{dim}(U)+\text{dim}(T)+\text{dim}(B)$. We say that $A/L$ has purely additive reduction if $g=\text{dim}(U)$, or equivalently, if $\text{dim}(T)=\text{dim}(B)=0$.

The following two theorems will be very useful in the next section.

\begin{theorem}\label{cmpotentiallygoodreduction}(see \cite[Lemma 2.4]{oortgoodandstablereductionofabelianvarieties})
Let $R$ be a complete discrete valuation ring with fraction field $L$ and perfect residue field. Let $A/L$ be a simple abelian variety with complex multiplication by $K$ over the field $L$. Then $A/L$ has either purely additive or good reduction. 
\end{theorem}

\begin{theorem}\label{theoremlorenzini1990}(see \cite[Proposition 2.7]{lorenzini1990})
Let $R$ be a complete discrete valuation ring with fraction field $L$ and algebraically closed residue field of characteristic $p > 5$. Let $C/L$ be a projective, smooth, and geometrically connected curve of genus $2$ with Jacobian $Jac(C)/L$. Assume that the Jacobian $Jac(C)/L$ has purely additive and potentially good reduction. Then $[M:L] \leq 10$.
\end{theorem}

 We now recall some background material on reduction of algebraic curves. The reader is referred to \cite[Chapter 10]{Liubook} for more information on this topic. A projective, connected, and reduced curve $C/\bar{k}$ over an algebraically closed field $\bar{k}$ is called stable if it has arithmetic genus greater or equal to $2$, its singular points are ordinary double points, and all of its irreducible components that are isomorphic to $\mathbb{P}^1_{\bar{k}}$  meet the other components in at least $3$ points. Let $R$ be a complete discrete valuation ring with fraction field $L$ and algebraically closed residue field, and let $C/L$ be a smooth, projective, and geometrically connected curve of genus $g \geq 2$. Recall that a stable model of $C/L$ is a proper and flat scheme $\mathcal{C}/R$ whose generic fiber is isomorphic to the curve $C/L$ and whose special fiber is a stable curve. We will say that $C/L$ has stable reduction if it has a model whose special fiber is a stable curve over the residue field. 

\begin{theorem}\label{theoremextension}(see \cite[Theorem 10.4.44]{Liubook})
Let $R$ be a complete discrete valuation ring with fraction field $L$ and algebraically closed residue field. Let $C/L$ be a smooth, projective, and geometrically connected curve of genus $g \geq 2$. Then there exists a (unique) finite extension $M/L$ such that the curve $C_M/M$ has stable reduction that has the following minimality property; for every other finite extension $N/L$ the base change $C_N/N$ has stable reduction if and only if $M \subseteq N$.
\end{theorem}

Using the same assumptions as in the preceding theorem, we denote the stable model of $C_M/M$ by $\mathcal{C}^{st}$ which is unique by \cite[Theorem 10.3.34]{Liubook}. Below we will refer to $\mathcal{C}^{st}$ as the stable model, but the reader should keep in mind that $\mathcal{C}^{st}$ is the stable model of $C_M/M$. If $C/L$ has potentially good reduction, then $\mathcal{C}^{st}$ is smooth. On the other hand, if $C/L$ does not have potentially good reduction but the Jacobian $Jac(C)/L$ of $C/L$ has potentially good reduction, then the special fiber of $\mathcal{C}^{st}$ is a union of two elliptic curves meeting at a single point (see the paragraph before \cite[Proposition 2]{liu1993}).

The following lemma will be very useful in the proofs of the next section because it will help us exclude certain reduction types of the special fiber of the minimal proper regular model of our curve.

\begin{lemma}\label{minimalregularmoderationalcurves}
Let $R$ be a discrete valuation ring with fraction field $L$ and algebraically closed residue field $k$. Let $C/L$ be a projective, smooth, and geometrically connected curve of genus $2$. Assume that the Jacobian $Jac(C)/L$ has purely additive reduction. Then the special fiber of the minimal proper regular model of $C/L$ contains only rational curves.
\end{lemma}
\begin{proof}
This is well-known to the experts. Nevertheless, we provide some details of the proof. Let $\mathcal{C}^{min}/R$ be the minimal proper regular model of $C/L$. By performing a sequence of blowups of closed points of $\mathcal{C}^{min}/R$ we can find a new regular model that has the properties $(1)$ and $(2)$ of \cite[Section 6]{neronmodelslorenzini}. Moreover, the condition $r=1$ of \cite[Theorem 6.1]{neronmodelslorenzini} is automatically satisfied because $g=2$ (see \cite[Proposition 9.5.1]{raynaud1990}). Since blowing up at a closed point does not introduce curves of genus bigger than $0$ and $Jac(C)/L$ has purely additive reduction by assumption, it follows from \cite[Theorem 6.1, Part (a)]{neronmodelslorenzini} that the special fiber of the minimal proper regular model of $C/L$ can only contain rational curves.
\end{proof}

\section{Genus 2 CM curves}\label{sectiongenus2cm}

In this section we prove Theorem \ref{genus2jacobiangcurvegoodreduction}. For the convenience of the reader, we have repeated the statements to be proved and we have split the proof of the theorem into parts.

\begin{remark}
Recall that a quartic CM field $K$ is a totally imaginary quadratic extension of a totally real quadratic number field. Since the degree of the cyclotomic field $\mathbb{Q}(\zeta_{m})$ over $\mathbb{Q}$ is $\phi(m)$, we find that if $K$ is a quartic CM field which contains a primitive root of unity of order $m$, then $m=2,3,4,5,6,8,10$ or $12$. Note that $\mathbb{Q}(\zeta_5) \cong \mathbb{Q}(\zeta_{10})$, $\mathbb{Q}(\zeta_8)$, and $\mathbb{Q}(\zeta_{12})$ are all quartic CM fields and that $\mathbb{Q}(\zeta_8)$, $\mathbb{Q}(\zeta_{10})$, and $\mathbb{Q}(\zeta_{12})$ are the only quartic CM fields that have $8$, $10$, and $12$ roots of unity, respectively. From the above discussion, it follows that the number of roots of unity for a quartic CM field is $2,4,6,8,10,$ or $12$.
\end{remark}

\begin{lemma}\label{excludetypes}
Let $R$ be a complete discrete valuation ring with fraction field $L$ of characteristic $0$ and algebraically closed residue field $k_L$ of characteristic $p>5$. Let $C/L$ be a projective, smooth, and geometrically connected curve of genus $2$ with Jacobian $Jac(C)/L$ that has complex multiplication over the field $L$. Assume that $C/L$ has potentially good reduction. Then $C/L$ cannot have reduction type $[II]$, $[V]$, or $[V^*]$.
\end{lemma}
\begin{proof}
If $C/L$ has reduction type $[II]$, then we see from \cite[Page 155]{namikawauenoclassification} that the special fiber of the minimal regular model of $C/L$ contains a curve of genus $1$. Therefore, Lemma \ref{minimalregularmoderationalcurves} implies that the Jacobian $Jac(C)/L$ cannot have purely additive reduction. However, this contradicts Theorem \ref{cmpotentiallygoodreduction}.

Assume now that $C/L$ has reduction type $[V]$ or $[V^*]$. We will find a contradiction. The idea is that after a cubic base extension, $C/L$ will acquire reduction of type $[II]$, so we get a contradiction by the previous paragraph. If $C/L$ has reduction of $[V]$ or $[V^*]$, then, since $p>5$, \cite[Table 1]{Liugenus2algorithm} tells us that $C/L$ acquires good reduction after a cyclic extension $M/L$ of degree $6$. Let now $N/L$ be the cubic field sub-extension of $M/L$. The base change $C_N/N$ acquires good reduction after an extension of degree $2$. Therefore, using \cite[Table 1]{Liugenus2algorithm}, we find that $C_N/N$ has either reduction of type $[II]$ or $[I_{0-0-0}^*]$. Let $r$ be the $r$-invariant defined by \cite[Théorème 1]{Liugenus2algorithm} corresponding to $C/L$ and let $r'$ be the the $r$-invariant defined by \cite[Théorème 1]{Liugenus2algorithm} corresponding to $C_N/N$. We will not introduce Liu's notation because we will only use it very briefly here, but the interested reader is referred to \cite{Liugenus2algorithm} or \cite[Section 5.1]{sadekminimalregularmodelsandtwists} for more information. It follows from \cite[Table 1]{Liugenus2algorithm} that $r \equiv 1 \text{ or } 5 \; (\text{mod } 6)$, because $n=6$ for $C/L$. Let $v_L$ and $v_N$ be the associated (normalized) valuations for $L$ and $N$ respectively. By looking at the expression for $r'$ and keeping in mind that $v_N|_L=3v_L$, we find that $r' \equiv 1\; (\text{mod } 2)$. Therefore, using \cite[Table 1]{Liugenus2algorithm}, we find that $C_N/N$ has reduction type $[II]$, which is a contradiction.
\end{proof}

We now prove Part $(i)$ of Theorem \ref{genus2jacobiangcurvegoodreduction}.

\begin{theorem}\label{potentiallygoodreductionproof}
Let $R$ be a complete discrete valuation ring with fraction field $L$ of characteristic $0$ and finite residue field $k_L$ of characteristic $p>5$. Let $C/L$ be a projective, smooth, and geometrically connected curve of genus $2$ with Jacobian $Jac(C)/L$ that has complex multiplication by a quartic CM field $K$ over the field $L$. Let $\mu'=|\mu(K)|$, where $\mu(K)$ are the roots of unity in $K$, and let $L^{unr}$ be the maximal unramified extension of $L$. Assume that $C/L$ has potentially good reduction. Then
\begin{enumerate}
    \item If $\mu'=2$, then $C_{L^{unr}}/L^{unr}$ can only have reduction of type $[I_{0-0-0}^*]$ or good reduction $[I_{0-0-0}]$.
    \item If $\mu'=4$, then $C_{L^{unr}}/L^{unr}$ can only have reduction of type $[I_{0-0-0}^*]$, $[VI]$, or good reduction $[I_{0-0-0}]$.
    \item If $\mu'=6$, then $C_{L^{unr}}/L^{unr}$ can only havr reduction of type $[I_{0-0-0}^*]$, $[III]$, $[IV]$, or good reduction $[I_{0-0-0}]$.
    \item If $\mu'=8$, then $C_{L^{unr}}/L^{unr}$ can only have reduction of type $[I_{0-0-0}^*]$, $[VI]$, $[VII]$, $[VII^*]$, or good reduction $[I_{0-0-0}]$.
    \item If $\mu'=10$, then $C_{L^{unr}}/L^{unr}$ can only have reduction of type $[I_{0-0-0}^*]$, $[IX-1]$, $[IX-2]$, $[IX-3]$, $[IX-4]$, $[VIII-1]$, $[VIII-2]$,$[VIII-3]$, $[VIII-4]$ or good reduction $[I_{0-0-0}]$.
    \item If $\mu'=12$, then $C_{L^{unr}}/L^{unr}$ can only have reduction of type $[I_{0-0-0}^*]$, $[III]$, $[IV]$, $[VI]$, or good reduction $[I_{0-0-0}]$.
\end{enumerate}
\end{theorem}

\begin{proof}
Let $M/L^{unr}$ be the extension of minimal degree over which $C_{L^{unr}}/L^{unr}$ acquires stable reduction, which is provided by Theorem \ref{theoremextension}. By \cite[Theorem (2.4)]{delignemumford} and Lemma \ref{lemmarootsofunity} we find that $[M : L^{unr} ]$ divides $\mu'=|\mu(K)|$. The idea of the proof is to use this divisibility combined with Liu's algorithm \cite{Liugenus2algorithm} and Lemma \ref{excludetypes} to compute the possible reduction types of $C_{L^{unr}}/L^{unr}$.

Since we assume that $C/L$ has potentially good reduction, it follows from \cite[Proposition 3]{liu1993} that the stable model $\mathcal{C}^{st}$ has good reduction. Moreover, since $p>5$, the extension $M/L^{unr}$ is tame (see \cite[Proposition 4.1.2]{Liugenus2algorithm}). Therefore, by \cite[Table 1]{Liugenus2algorithm} we find that the possible reduction types of $C_{L^{unr}}/L^{unr}$ are as follows

\begin{center}
\begin{tabular}{|c| c| }
\hline 
  Degree of $M/L^{unr}$  & Possible reduction types of $C_{L^{unr}}/L^{unr}$  \\
    \hline 
    1 & $[I_{0-0-0}]$ \\\hline
    2 & $[I_{0-0-0}^*]$, $[II]$ \\\hline
    3 & $[III]$\\\hline
    4 & $[VI]$\\\hline
    5 & $[IX-1]$, $[IX-2]$, $[IX-3]$, $[IX-4]$\\\hline
    6 & $[IV]$, $[V]$, $[V^*]$\\\hline
    8 & $[VII]$, $[VII^*]$\\\hline
    10 & $[VIII-1]$, $[VIII-2]$,$[VIII-3]$, $[VIII-4]$\\
    \hline 

    \hline
\end{tabular}
 \end{center}

{\it Proof of part $(i)$:} If $\mu'=2$, then we see that $[M : L^{unr} ]$ is either $1$ or $2$. Therefore, the curve $C_{L^{unr}}/L^{unr}$ has either reduction of type $[I_{0-0-0}]$, $[I_{0-0-0}^*]$, or $[II]$. However, Lemma \ref{excludetypes} excludes the case of type $[II]$.

{\it Proof of part $(ii)$:} If $\mu'=4$, then we see that $[M : L^{unr} ]$ divides $4$. If $[M: L^{unr}] \leq 2$, then we find that the reduction types of $C_{L^{unr}}/L^{unr}$ are the types that appear in part $(i)$. On the other hand, if $[M: L^{unr}] = 4$, then $C_{L^{unr}}/L^{unr}$ has reduction of type $[VI]$.

{\it Proof of part $(iii)$:} If $\mu'=6$, then we see that $[M : L^{unr} ]$ divides $6$. Therefore, $[M : L^{unr} ]=1,2,3,$ or $6$. Using the table above we find that $C_{L^{unr}}/L^{unr}$ has reduction of type $[I_{0-0-0}^*]$, $[II]$, $[III]$, $[IV]$, $[V]$, or $[V^*]$. However, using Lemma \ref{excludetypes}, we find that the types $[II]$, $[V]$, and $[V^*]$ cannot occur.

{\it Proof of part $(iv)$:} If $\mu'=8$, then we see that $[M : L^{unr} ]$ divides $8$. If $[M: L^{unr}]$ divides $4$,  then we find that the reduction types of $C_{L^{unr}}/L^{unr}$ are the types that appear in part $(ii)$. On the other hand, if $[M : L^{unr}] = 8$, then we find that $C_{L^{unr}}/L^{unr}$ has reduction of type $[VII]$ or $[VII^*]$.

{\it Proof of part $(v)$:} If $\mu'=10$, then we see that $[M : L^{unr} ]$ divides $10$. If $[M : L^{unr}]=1$ or $2$, then $C_{L^{unr}}/L^{unr}$ has reduction of type $[I_{0-0-0}^*]$ or $[I_{0-0-0}]$. On the other hand, if $[M : L^{unr}]=5$ or $10$, then we find that $C_{L^{unr}}/L^{unr}$ has reduction of type $[IX-1]$, $[IX-2]$, $[IX-3]$, $[IX-4]$, $[VIII-1]$, $[VIII-2]$,$[VIII-3]$, $[VIII-4]$.

{\it Proof of part $(vi)$:} If $\mu'=12$, the we see that $[M : L^{unr} ]$ divides $12$. Since $[M : L^{unr} ] \leq 10$ by Theorem \ref{theoremlorenzini1990}, we find that $[M : L^{unr} ] = 1,2,3,4,$ or $6$. If $[M: L^{unr} ]$ divides $6$, then we find that the reduction types of $C_{L^{unr}}/L^{unr}$ are the types that appear in part $(iii)$. On the other hand, if $[M : L^{unr} ]=4$, then $C_{L^{unr}}/L^{unr}$ can only have reduction of type $[VI]$. This proves our theorem.
\end{proof}

\begin{emptyremark}\label{c0c1}
The following two curves, $C_0/k$ and $C_1/k$, will play a special role in Theorem \ref{theorem4.5} below. Let $k$ be an algebraically closed field of characteristic $p>5$. Let $C_0/k$ be the smooth projective geometrically connected curve given by the following affine equation $$C_0 \; : \; y^2=x^5-1.$$
Moreover, let $C_1/k$ be the smooth projective geometrically connected curve given by the following affine equation $$C_1 \; : \; y^2=x^5-x.$$
\end{emptyremark}

We now show that under some extra assumptions on the special fiber of the stable model of $C/L$, we can achieve more precise results.

\begin{theorem}\label{theorem4.5}
Keep the same assumptions and notation as in Theorem \ref{potentiallygoodreductionproof} and assume in addition that the special fiber $C_s$ of $\mathcal{C}^{st}$ is not isomorphic to either the curve $C_0/k$ or the curve $C_1/k$ of \ref{c0c1}. Then
\begin{enumerate}
    \item If $\mu'=8$, then $C_{L^{unr}}/L^{unr}$ can only have reduction of type $[I_{0-0-0}^*]$, $[VI]$, or good reduction $[I_{0-0-0}]$.
    \item If $\mu'=10$, then $C_{L^{unr}}/L^{unr}$ can only have reduction of type $[I_{0-0-0}^*]$ or good reduction $[I_{0-0-0}]$.
\end{enumerate}
\end{theorem}

\begin{proof}
 Let $M/L^{unr}$ be the extension of minimal degree over which $C_{L^{unr}}/L^{unr}$ acquires stable reduction, which is provided by Theorem \ref{theoremextension}. Since $C/L$ has potentially good reduction, it follows from \cite[Corollaire 4.1]{liu1993} that $[M:L^{unr}]$ divides $4$ or $6$, except if $C_s \cong C_0, C_1$, where $C_s$ is the stable reduction of $C/L$.
 
 {\it Proof of part $(i)$:} Assume that $\mu'=8$. Lemma \ref{lemmarootsofunity} implies that $[M : L^{unr} ]$ divides $8$. Since $[M:L^{unr}]$ also divides $4$ or $6$, we find that $[M : L^{unr} ]=1,2,$ or $4$. Therefore, proceeding in a similar way as in the proof of part $(ii)$ of Theorem \ref{potentiallygoodreductionproof}, we find that $C_{L^{unr}}/L^{unr}$ can only have reduction of type $[I_{0-0-0}^*]$, $[VI]$, or good reduction $[I_{0-0-0}]$.
 
 {\it Proof of part $(ii)$:} Assume that $\mu'=10$. Lemma \ref{lemmarootsofunity} implies that $[M : L^{unr} ]$ divides $10$. Since $[M:L^{unr}]$ also divides $4$ or $6$, we find that $[M : L^{unr} ]=1$ or $2$. Therefore, proceeding in a similar way as in the proof of part $(i)$ of Theorem \ref{potentiallygoodreductionproof}, we find that $C_{L^{unr}}/L^{unr}$ can only have reduction of type $[I_{0-0-0}^*]$ or good reduction $[I_{0-0-0}]$. This proves our theorem.
\end{proof}

We now consider the possible types that can occur when $C/L$ does not have potentially good reduction. 

\begin{emptyremark}\label{remarkdegreeofsingularity}
Let $R$ be a complete discrete valuation ring with valuation $v_L$, fraction field $L$ of characteristic $0$, and algebraically closed residue field $k_L$ of characteristic $p>5$. Let $C/L$ be a projective, smooth, and geometrically connected curve of genus $2$ with simple Jacobian $Jac(C)/L$ that has complex multiplication by a quartic CM field $K$ over the field $L$. Assume also that $C/L$ does not have potentially good reduction. Let $M/L$ be the extension of minimal degree over which $C/L$ acquires stable reduction, which exits by Theorem \ref{theoremextension}. Recall that we denote by $\mathcal{C}^{st}$ the stable model of $C_M/M$. It follows that the special fiber of the stable model $\mathcal{C}^{st}$ is a union of two elliptic curves $E_1$ and $E_2$ intersecting at a point, see the paragraph before \cite[Proposition 2]{liu1993} and note that the Jacobian of $C/L$ has potentially good reduction. Let $d_L:=v_L(J_{10}J_2^{-5})/12$, where $J_{10}$ and $J_2$ are the (Igusa) $J_{10}$- and $J_2$-invariants associated to $C/L$, see \cite[Section 2.2]{Liugenus2algorithm} for the relevant definitions. The number $d:=[M:L]d_L$ is called the degree of singularity of the point of intersection $E_1 \cap E_2$ in $\mathcal{C}^{st}$. We note that this $d$ is the same $d$ that appears in the reduction type of the second part of Theorem \ref{genus2jacobiangcurvegoodreduction} (as well as in its restatement below).
\end{emptyremark}

Before we proceed to the proof of the second part of Theorem \ref{genus2jacobiangcurvegoodreduction}, we need to prove a lemma that will significantly simplify our proof. We note that in the lemma and the theorem below the number $r$ corresponds to the $r$-invariant defined in \cite[Théorème 3]{Liugenus2algorithm}.

\begin{lemma}\label{excludetypesII}
Let $R$ be a complete discrete valuation ring with (normalized) valuation $v_L$, with fraction field $L$ of characteristic $0$, and algebraically closed residue field $k_L$ of characteristic $p>5$. Let $C/L$ be a projective, smooth, and geometrically connected curve of genus $2$ with Jacobian $Jac(C)/L$ that has complex multiplication over the field $L$. 
\begin{enumerate}
    \item The curve $C/L$ cannot have any of the following reduction types

\begin{center}
\begin{tabular}{|c| c| c |}
\hline
  $[I_0-I_0-d]$ & $[I_0-I_0^*-(d-1)/2]$ & $[2I_0-r]$ \\
    \hline
    $[I_0-IV-(d-1)/3]$ & $[I_0-IV^*-(d-2)/3]$ & $[I_0-III-(d-1)/4]$\\
    \hline 
    $[I_0-III^*-(d-3)/4]$ & $[I_0^*-III^*-(d-5)/4]$ & $[I_0^*-III-(d-3)/4]$\\
    \hline
    $[2IV-(r-1)/3]$ & $[2IV^*-(r-2)/3]$ & $[I_0-II-(d-1)/6]$ \\
    \hline
    $[I_0-II^*-(d-5)/6]$ & $[I_0^*-IV^*-(d-7)/6]$ & $[I_0^*-IV-(d-5)/6]$ \\
    \hline
    $[II^*-IV-(d-7)/6]$ & $[II-IV-(d-3)/6]$ & $[II-IV^*-(d-5)/6]$\\
    \hline
    $[II^*-IV^*-(d-9)/6]$ &  & \\
    \hline
\end{tabular}
    
\end{center} 

\item If $C/L$ acquires semi-stable reduction after an extension of degree $6$ or $12$, then $2 \nmid v_L(J_2)$ except if $C/L$ has reduction type $[II-II-(d-2)/6]$, $[II-II^*-(d-6)/6]$, $[II^*-II^*-(d-10)/6]$, $[I_0^*-II-(d-4)/6]$, or $[I_0^*-II^*-(d-8)/6]$.
\end{enumerate}
\end{lemma}
\begin{proof}
{\it Proof of part $(i)$:}
If $C/L$ has reduction of type $[I_0-I_0-d]$, $[I_0-I_0^*-(d-1)/2]$, $[2I_0-r]$, $[I_0-IV-(d-1)/3]$, or $[I_0-IV^*-(d-2)/3]$ then we see from \cite{namikawauenoclassification} that the special fiber of the minimal regular model of $C/L$ contains a smooth curve of genus $1$. Therefore, Lemma \ref{minimalregularmoderationalcurves} implies that the Jacobian $Jac(C)/L$ cannot have purely additive reduction. However, this contradicts Theorem \ref{cmpotentiallygoodreduction}.

Assume that $C/L$ has reduction of type $[I_0-III-(d-1)/4]$, $[I_0-III^*-(d-3)/4]$, $[I_0^*-III^*-(d-5)/4]$, or $[I_0^*-III-(d-3)/4]$, and we will find a contradiction. We know from \cite[Table 3.1]{Liugenus2algorithm} that $C/L$ acquires stable reduction after a cyclic extension $M/L$ of degree $4$. Let now $N/L$ be a quadratic field extension contained in $M/L$. The base change $C_N/N$ acquires stable reduction after an extension of degree $2$. Let $d$ be the integer defined in \ref{remarkdegreeofsingularity} computed with respect to the field $L$. Looking at \cite[Table 3.1]{Liugenus2algorithm}, we find that $d \equiv 1 \text{ or } 3 \; (\text{mod } 4)$.  Let $v_L$ and $v_N$ are the corresponding valuations for $L$ and $N$ respectively. Since $v_N|_L=2v_L$ and $[M:N]=[M:N]/2$, we see that the value of $d$ remains the same when computed over $N$. Therefore, using \cite[Table 3.1]{Liugenus2algorithm}, we see that $C_N/N$ has reduction $[I_0-I_0^*-(d-1)/2]$, because $d \equiv 1 \; (\text{mod } 2)$. This contradicts the previous paragraph because $Jac(C_N)/N$ has complex multiplication over $N$.

Assume that $C/L$ has reduction of type $[2IV-(r-1)/3]$ or $[2IV^*-(r-2)/3]$, and we will find a contradiction. Let $v_L$ be the (normalized) valuation of $L$ and let $J_2$ be the $J_2$-invariant associated to $C/L$ (see \cite[Section 2.2]{Liugenus2algorithm}). Since $C/L$ has reduction of type $[2IV-(r-1)/3]$ or $[2IV^*-(r-2)/3]$, we know that $2 \nmid v_L(J_2)$ by \cite[Table 3.2]{Liugenus2algorithm}. Moreover, \cite[Table 3.1]{Liugenus2algorithm} tells us that $C/L$ acquires good reduction after a cyclic extension $M/L$ of degree $6$. Let now $N/L$ be a cubic extension contained in $M/L$ and denote by $v_N$ the corresponding valuation. Since $v_N|_L=3v_L$, we see that $2 \nmid v_N(J_2)$. Therefore, $C_N/N$ must have reduction of type $[2I_0-r]$, by \cite[Table 3.2]{Liugenus2algorithm}. This is a contradiction, by Lemma \ref{excludetypesII}, because $Jac(C_N)/N$ has complex multiplication over $N$.

Assume that $C/L$ has reduction of type  $[I_0-II-(d-1)/6]$,  $[I_0-II^*-(d-5)/6]$, $[I_0^*-IV^*-(d-7)/6]$, $[I_0^*-IV-(d-5)/6]$, $[II^*-IV-(d-7)/6]$, $[II-IV-(d-3)/6]$, $[II-IV^*-(d-5)/6]$, or $[II^*-IV^*-(d-9)/6]$, and we will find a contradiction. We know from \cite[Table 3.1]{Liugenus2algorithm} that $C/L$ acquires good reduction after a cyclic extension $M/L$ of degree $6$. Let $d$ be the integer defined in \ref{remarkdegreeofsingularity} with respect to $L$. Looking at \cite[Table 3.1]{Liugenus2algorithm}, we find that $d \equiv 1,3, \text{ or } 5 \; (\text{mod } 6 )$. Let now $N/L$ be a cubic extension contained in $M/L$ and denote by $v_N$ the corresponding valuation. Since $v_N|_L=3v_L$ and $[M:N]=[M:N]/3$, we see that the value of $d$ remains the same when computed over $N$. Therefore, using \cite[Table 3.1]{Liugenus2algorithm}, we see that $C_N/N$ has reduction $[I_0-I_0^*-(d-1)/2]$, because $d \equiv 1 \; (\text{mod } 2)$. Finally, using Lemma \ref{excludetypesII} we find a contradiction because $Jac(C_N)/N$ has complex multiplication over $N$.

{\it Proof of part $(ii)$:}
Assume that $2 \mid v_L(J_2)$ and that $C/L$ does not have reduction type $[II-II-(d-2)/6]$, $[II-II^*-(d-6)/6]$, $[II^*-II^*-(d-10)/6]$, $[I_0^*-II-(d-4)/6]$, or $[I_0^*-II^*-(d-8)/6]$. We proceed by contradiction. Let $M/L$ be the extension of minimal degree over which $C/L$ acquires stable reduction. Note that the special fiber of the stable model of $C_M/M$ is a union of two elliptic curves meeting at one point (see the paragraph right before \cite[Proposition 2]{liu1993}). By assumption $[M:L]=6$ or $12$. Let now $N/L$ be the extension of degree $3$ contained in $M/L$ and denote by $v_N$ the corresponding valuation. Then $C_N/N$ acquires good reduction after an extension of degree $[M:L]/3$ and $2 \mid v_N(J_2)$. Let $d$ be the integer defined in Paragraph \ref{remarkdegreeofsingularity} with respect to $L$. It follows from \cite[Table 3.1]{Liugenus2algorithm} that either $d \equiv 1,3, \text{ or } 5 \; (\text{mod } 6 )$ or $d \equiv 1,3, \text{ or } 5 \; (\text{mod } 6 )$. The curve $C_N/N$ has stable reduction after an extension of degree $2$ or $4$ while $d$ is odd. Let $v_N$ be the valuation of $N$. Since $v_N|_L=([M:L]/3)v_L$ and $[M:N]=3$, we see that the value of $d$ remains the same when computed over $N$.  Therefore, using \cite[Table 3.1]{Liugenus2algorithm}, we see that $C_N/N$ has reduction type $[I_0-I_0^*-(d-1)/2]$, $[I_0-III-(d-1)/4]$, $[I_0-III^*-(d-3)/4]$, $[I_0^*-III-(d-3)/4]$, or $[I_0^*-III^*-(d-5)/4]$. This is a contradiction by Lemma \ref{excludetypesII} because $Jac(C_N)$ has complex multiplication over the field $N$. This proves our lemma.
\end{proof}

We now prove the second part of Theorem \ref{genus2jacobiangcurvegoodreduction}. For the convenience of the reader we have repeated the statement to be proved.

\begin{theorem}\label{notpotentiallygoodgenus2}
Let $R$ be a complete discrete valuation ring with fraction field $L$ of characteristic $0$ and finite residue field $k_L$ of characteristic $p>5$. Let $C/L$ be a projective, smooth, and geometrically connected curve of genus $2$ with Jacobian $Jac(C)/L$ that has complex multiplication by a quartic CM field $K$ over the field $L$.  Let $\mu'=|\mu(K)|$, where $\mu(K)$ is the group of roots of unity in $K$. Assume that $C/L$ does not have potentially good reduction. Then
\begin{enumerate}[resume]
    \item If $\mu'=2$ or $10$, then $C_{L^{unr}}/L^{unr}$ can only have reduction of type $[I_0^*-I_0^*-(d-2)/2]$.
    \item If $\mu'=4$, then $C_{L^{unr}}/L^{unr}$ can only have reduction of type $[I_0^*-I_0^*-(d-2)/2]$, $[III-III-(d-2)/4]$, $[III-III^*-(d-4)/4]$, $[III^*-III^*-(d-6)/4]$, or $[2I_0^*-(r-1)/2]$.
    \item If $\mu'=6$, then $C_{L^{unr}}/L^{unr}$ can only have reduction of type $[I_0^*-I_0^*-(d-2)/2]$, $[IV-IV-(d-2)/3]$, $[IV-IV^*-(d-3)/3]$, $[IV^*-IV^*-(d-4)/3]$, $[II-II-(d-2)/6]$, $[II-II^*-(d-6)/6]$, $[II^*-II^*-(d-10)/6]$, $[I_0^*-II-(d-4)/6]$,
    $[I_0^*-II^*-(d-8)/6]$, $[2IV-(r-1)/3]$, or $[2IV^*-(r-2)/3]$.
    \item If $\mu'=8$, then $C_{L^{unr}}/L^{unr}$ can only have reduction of type $[I_0^*-I_0^*-(d-2)/2]$, $[III-III-(d-2)/4]$, $[III-III^*-(d-4)/4]$, $[III^*-III^*-(d-6)/4]$, $[2I_0^*-(r-1)/2]$, $[2III-(r-1)/4]$, $[2III^*-(r-3)/4]$.
    \item If $\mu'=12$, then the possible reduction types of $C_{L^{unr}}/L^{unr}$ are given by the following table 
    \begin{center}
\begin{tabular}{|c| c| c |}
\hline
  $[I_0^*-I_0^*-(d-2)/2]$ & $[III-III-(d-2)/4]$ & $[III-III^*-(d-4)/4]$ \\
    \hline
    $[III^*-III^*-(d-6)/4]$ & $[2I_0^*-(r-1)/2]$ & $[IV-IV-(d-2)/3]$\\
    \hline 
    $[IV-IV^*-(d-3)/3]$ & $[IV^*-IV^*-(d-4)/3]$ & $[II-II-(d-2)/6]$,\\
    \hline
    $[II-II^*-(d-6)/6]$ & $[II^*-II^*-(d-10)/6]$ & $[I_0^*-II-(d-4)/6]$,\\
    \hline 
    $[I_0^*-II^*-(d-8)/6]$& $[2IV-(r-1)/3$ & $[2IV^*-(r-2)/3]$ \\
    \hline
     $[2II-(r-1)/6]$ & $[2II^*-(r-5)/6]$ & \\
     \hline

\end{tabular}
    
\end{center} 
\end{enumerate}
\end{theorem}

\begin{proof}
Let $M/L^{unr}$ be the extension of minimal degree over which $C_{L^{unr}}/L^{unr}$ acquires stable reduction, which is provided by Theorem \ref{theoremextension}. By \cite[Theorem (2.4)]{delignemumford} and Lemma \ref{lemmarootsofunity} we find that $[M : L^{unr} ]$ divides $\mu'=|\mu(K)|$. Moreover, since $p>5$, the extension $M/L^{unr}$ is tame (see \cite[Proposition 4.1.2]{Liugenus2algorithm}). It follows from \cite[Proposition 3]{liu1993} that $C_M/M$ has stable reduction. Moreover, since the stable model of $C_M/M$ is assumed to be singular, $[M:L^{unr}]$ divides $8$ or $12$ by \cite[Corollaire 4.1]{liu1993} and the special fiber of the stable model of $C_M/M$ is a union of two elliptic curves meeting at one point (see the paragraph right before \cite[Proposition 2]{liu1993}). We will combine the above observations with Liu's algorithm \cite{Liugenus2algorithm} along with Lemma \ref{excludetypesII} to compute the special fiber of the minimal proper regular model.

{\it Proof of part $(i)$:} If $\mu'=2$ or $10$, then we see that $[M : L^{unr} ]$ divides $2$ or $10$. Since $[M:L^{unr}]$ also divides $8$ or $12$, we find that $[M : L^{unr} ]=1$ or $2$. Therefore, we find, using \cite[Table 3.1 and Table 3.2]{Liugenus2algorithm}, that $C_{L^{unr}}/L^{unr}$ has either reduction of type $[I_0^*-I_0^*-(d-2)/2]$, $[I_0-I_0-d]$, $[I_0-I_0^*-(d-1)/2]$, or $[2I_0-r]$. However, the last three types are excluded by Part $(i)$ of Lemma \ref{excludetypesII}.

{\it Proof of part $(ii)$:} If $\mu'=4$, then we see that $[M : L^{unr} ]$ divides $4$. If $[M : L^{unr} ]=1$ or $2$,then we find that the reduction types of $C_{L^{unr}}/L^{unr}$ are the types that appear in part $(i)$. Therefore, we assume from now on that $[M: L^{unr} ] = 4$. Using \cite[Table 3.1]{Liugenus2algorithm} and \cite[Table 3.2]{Liugenus2algorithm}, we find that $C_{L^{unr}}/L^{unr}$ can only have reduction of type $[I_0^*-I_0^*-(d-2)/2]$, $[III-III-(d-2)/4]$, $[III-III^*-(d-4)/4]$, $[III^*-III^*-(d-6)/4]$,  $[2I_0^*-(r-1)/2]$, $[I_0-III-(d-1)/4]$, $[I_0-III^*-(d-3)/4]$, $[I_0^*-III^*-(d-5)/4]$, $[I_0^*-III-(d-3)/4]$. However, the last four types are excluded by Part $(i)$ of Lemma \ref{excludetypesII}.

{\it Proof of part $(iii)$:} If $\mu'=6$, then we see that $[M : L^{unr} ]$ divides $6$. If $[M : L^{unr} ]=1$ or $2$, then we find that the reduction types of $C_{L^{unr}}/L^{unr}$ are the types that appear in part $(i)$. If $[M : L^{unr} ]=3$,  then, using \cite[Table 3.1]{Liugenus2algorithm}, we find that $C_{L^{unr}}/L^{unr}$ can only have reduction of type $[IV-IV-(d-2)/3]$, $[IV-IV^*-(d-3)/3]$, $[IV^*-IV^*-(d-4)/3]$, $[I_0-IV-(d-1)/3]$, or $[I_0-IV^*-(d-2)/3]$.  However, the last two types are excluded by Lemma \ref{excludetypesII}. On the other hand, if $[M : L^{unr} ] = 6$, then, Part $(ii)$ of Lemma \ref{excludetypesII} below shows that either $C_{L^{unr}}/L^{unr}$ has reduction type $[II-II-(d-2)/6]$, $[II-II^*-(d-6)/6]$, $[II^*-II^*-(d-10)/6]$, $[I_0^*-II-(d-4)/6]$,$[I_0^*-II^*-(d-8)/6]$, or we have that $2 \nmid v_L(J_2)$. Assume from now on that $2 \nmid v_L(J_2)$. Using \cite[Table 3.2]{Liugenus2algorithm}, we find that $C_{L^{unr}}/L^{unr}$ can only have reduction of type $[2IV-(r-1)/3]$ or $[2IV^*-(r-2)/3]$.

{\it Proof of part $(iv)$:} If $\mu'=8$, then we see that $[M : L^{unr} ]$ divides $8$. If $[M : L^{unr} ]$ divides $4$, then we find that the reduction types of $C_{L^{unr}}/L^{unr}$ are the types that appear in part $(ii)$. On the other hand, if $[M : L^{unr} ] = 8$, then, using \cite[Table 3.1]{Liugenus2algorithm} and \cite[Table 3.2]{Liugenus2algorithm}, we find that $C_{L^{unr}}/L^{unr}$ can only have reduction of type $[2III-(r-1)/4]$ or $[2III^*-(r-3)/4]$.

{\it Proof of part $(v)$:} If $\mu'=12$, then we see that $[M : L^{unr} ]$ divides $12$. If $[M : L^{unr} ]$ divides $4$ or $6$, then we find that the reduction types of $C_{L^{unr}}/L^{unr}$ are the types that appear in part $(ii)$ or in part $(iii)$. Assume now that $[M : L^{unr} ] = 12$. Part $(ii)$ of Lemma \ref{excludetypesII} shows that $2 \nmid v_L(J_2)$. Therefore, using  \cite[Table 3.2]{Liugenus2algorithm}, we find that $C_{L^{unr}}/L^{unr}$ can only have reduction of type  $[2II-(r-1)/6]$ or $[2II^*-(r-5)/6]$. This completes the proof of our theorem.
\end{proof}

\begin{remark}
 In Theorem \ref{genus2jacobiangcurvegoodreduction}, if a reduction type appears in the tables, then the reduction types corresponding to quadratic twists of the original curve also appear in the tables, as we explain in this remark. Keep the same assumptions and  notation as in Theorem \ref{genus2jacobiangcurvegoodreduction}. Let $L_d=L(\sqrt{d})$ be a quadratic extension of $L$ and let $\chi: \text{Gal} ( \widebar{L}/L)\rightarrow \{ \pm 1 \}$ with $\chi ( \sigma) = (\sqrt{d})^{\sigma}/\sqrt{d}$ be the associated quadratic character. Denote by $j$ the hyperelliptic involution of $C/L$. Consider the cocycle $\xi \in \text{H}^1( \text{Gal} ( \widebar{L}/L), \text{Aut}_{\widebar{L}}(C) ) $ given by $\xi(\sigma)=[j]$ if $\chi(\sigma)=-1$, and $\xi(\sigma)=[id]$ otherwise. Let $C^\xi/L$ be the twist of $C/L$ corresponding to the cocycle $\xi$. We claim that $Jac(C^{\xi})/L$ also has complex multiplication by $K$ over the field $L$. To justify this, note that the embedding $\iota : K \hookrightarrow \text{End}_L^0(Jac(C))$ induces an embedding $\iota_{\xi} : K \hookrightarrow \text{End}_L^0(Jac(C^{\xi}))$ (the proof of \cite[Lemma 2.2]{shnidmanquadraticrm} carries over in our case). We note that Sadek in \cite{sadekminimalregularmodelsandtwists} has computed the reduction type of $C^\xi/L$ based on the reduction type of $C/L$.
\end{remark}

We do not know whether all the types allowed by Theorem \ref{genus2jacobiangcurvegoodreduction} actually occur. We present some examples of reduction types of CM curves below.

\begin{example}(See \cite[Example 3.6.2]{gorenlauter} and \cite[Table 1]{vanwamelen})
Consider the hyperelliptic curve $C/\mathbb{Q}$ given by the following equation $$ \; y^2=-8x^6-64x^5+1120x^4+4760x^3-48400x^2+22627x-91839.$$ The Jacobian of this curve does not have complex multiplication over $\mathbb{Q}$ but it acquires complex multiplication after a finite extension of $\mathbb{Q}$.

Using SAGE \cite{sagemath} we see that the curve $C/\mathbb{Q}$ has bad reduction modulo $2,5,11,$ and $13$, it has reduction type $[I_0^*-I_0^*-0]$ modulo $11$, and it has reduction type $[VI]$ modulo $13$. Therefore, $C/\mathbb{Q}$ has potentially good reduction modulo $13$, and the potential stable reduction of $C/\mathbb{Q}$ modulo $11$ is the union of two elliptic curves intersecting transversely at a point. Let $N>1$ be a positive integer not divisible by $2,5,11,$ or $13$, and let $L=\mathbb{Q}(Jac(C))[N])$. Then $Jac(C)_L/L$ has good reduction, by \cite[Corollary 3]{st}, and it has all of its endomorphisms defined over $L$, by \cite[Theorem 2.4]{silverberg1992}. Therefore, $C_L/L$ has good reduction modulo primes of $L$ above $13$ and it has reduction of type $[I_0-I_0-2]$ modulo primes of $L$ above $11$.
\end{example}

\begin{example}(See \cite[Example 3.2]{gorenlauter} and \cite[Table 1]{vanwamelen}) 
Consider the hyperelliptic curve $C/\mathbb{Q}$ given by the following equation $$ \; y^2=x^5+1.$$ The curve $C/\mathbb{Q}$ has good reduction outside of $2,5$ and it has reduction type $[VII]$ modulo $5$. Moreover, the base extension $\text{Jac}(C)_{\mathbb{Q}(\zeta_5)}/\mathbb{Q}(\zeta_5)$ of the Jacobian of $C/\mathbb{Q}$ to $\mathbb{Q}(\zeta_5)$ has complex multiplication by $\mathbb{Q}(\zeta_5)$.
\end{example}

\section{Geometric component groups and torsion of CM abelian varieties}\label{sectioncomponentgroups}

In this section, we focus on component groups and torsion points of CM abelian varieties. We first prove Proposition \ref{componentgroups} below which provides a bound for the component group of CM abelian varieties and a list of the possible component groups of elliptic curves with complex multiplication. We then prove Proposition \ref{propositioncmandtorsion} using Proposition \ref{componentgroups}.

Let $R$ be a complete discrete valuation ring with fraction field $L$ of characteristic $0$ and perfect residue field $k_L$ of characteristic $p>0$, and let $A/K$ be an abelian variety. We denote by $\mathcal{A}/R$ the N\'eron model of $A/K$ (see \cite{neronmodelsbook} for the definition as well as the basic properties of N\'eron models). The special fiber $\mathcal{A}_{k_L}/k_L$ of $\mathcal{A}/R$ is a smooth commutative group scheme. We denote by $\mathcal{A}^0_{k_L}/k_L$ the connected component of the identity of $\mathcal{A}_{k_L}/k_L$. The finite \'etale group scheme defined by $\Phi:=\mathcal{A}_{k_L}/\mathcal{A}^0_{k_L}$ is called the component group of $\mathcal{A}/R$.

\begin{proposition}\label{componentgroups}
Let $R$ be a complete discrete valuation ring with fraction field $L$ of characteristic $0$ and finite residue field $k_L$ of characteristic $p$. Let $L^{unr}$ be the maximal unramified extension of $L$ and denote the residue field of $L^{unr}$ by $\widebar{k_L}$.
\begin{enumerate}
    \item Assume that $p \neq 2$ and that $K \neq \mathbb{Q}(i), \mathbb{Q}(\sqrt{-3})$. Let $E/L$ be an elliptic curve with complex multiplication by $K$ over $L$ such that  $j_E \neq 0, 1728$. Then $E/L$ has geometric component group $\Phi(\widebar{k_L})$ isomorphic (as an abelian group) to $(0)$ or $\mathbb{Z}/2\mathbb{Z} \times \mathbb{Z}/2\mathbb{Z}$.
    \item Let $A/L$ be an abelian variety with complex multiplication by a CM field $K$ over $L$. Then the geometric component group $\Phi(\widebar{k_L})$ of $A_{L^{unr}}/L^{unr}$  is killed by $|\mu(K)|$, where $\mu(K)$ is the group of roots of unity in $K$.
\end{enumerate}
\end{proposition}

\begin{proof}
{\it Proof of part $(i)$:} Theorem \ref{fullendomorphism} implies that $E/L$ can only have good reduction or reduction of Kodaira type I$_0^*$. Therefore, part $(i)$ follows from a simple application of Tate's algorithm (see \cite[Page 365]{silverman2}).

{\it Proof of part $(ii)$:} Let $M/L^{unr}$ be the extension of minimal degree over which $A_{L^{unr}}/L^{unr}$ acquires semi-stable reduction. It follows from Lemma \ref{lemmarootsofunity} that the degree $[M:L^{unr}]$ divides $|\mu(K)|$. On the other hand, work of McCallum \cite{mccallumthecomponentgroupofaneronmodel} and Edixhoven, Liu, and Lorenzini \cite[Theorem 1]{ell1996} tells us that $[M:L^{unr}]$ kills $\Phi(\widebar{k_L})$. Therefore, since $[M:L^{unr}]$ divides $|\mu(K)|$, we find that $\Phi(\widebar{k_L})$ is killed by $|\mu(K)|$. This proves our proposition.
\end{proof}

\begin{remark}
 Let $R$ be a complete discrete valuation ring with fraction field $L$ of characteristic $0$ and finite residue field $k_L$ of characteristic $p>0$. Lorenzini in \cite[Corollary 3.25]{lorenzini1993} has provided a list of the possible prime-to-$p$ parts of the geometric component group of abelian varieties defined over $L$ that have purely additive and potentially good reduction. When $p>5$ using Theorem \ref{genus2jacobiangcurvegoodreduction}, together with \cite[Section 8]{Liugenus2algorithm}, we find that the cases where the group is $\mathbb{Z}/4\mathbb{Z}$ or $\mathbb{Z}/2\mathbb{Z} \times \mathbb{Z}/4\mathbb{Z}$ cannot occur among Jacobian surfaces with complex multiplication defined over the base field. We note that by Theorem \ref{cmpotentiallygoodreduction} Jacobian surfaces with complex multiplication have either purely additive or good reduction.
\end{remark}

We now proceed to the proof of Proposition \ref{propositioncmandtorsion}. For every positive integer $m$ we let $\gamma_p(m)=\lfloor \log_p(\frac{pm}{p-1}) \rfloor.$

\begin{proposition}\label{torsiononcmbadreduction}
Let $R$ be a complete discrete valuation ring with fraction field $L$ of characteristic $0$ and finite residue field $k_L$ of characteristic $p>0$. Denote by $e$ the absolute ramification index of $L$. Let $A/L$ be an abelian variety with complex multiplication by a CM field $K$ over $L$ and assume that $A/L$ does not have good reduction. Then 
$$|A(L)_{tors}| \leq |\mu(K)| \cdot p^{2g \gamma_p(e|\mu(K)|)}.$$
\end{proposition}

\begin{proof}
By Theorem \ref{cmpotentiallygoodreduction} we find that $A/K$ has purely additive reduction. Then the proof of \cite[Part (iii) of Main Theorem]{cx} carries over verbatim in our case. The only extra input is that instead of using \cite[Corollary 3.8]{cx} in the proof of \cite[Part (iii) of Main Theorem]{cx} is that we can use the more precise bound provided by Proposition \ref{componentgroups}.
\end{proof}

Combining the previous proposition with \cite[Main Theorem]{cx} we prove Proposition \ref{propositioncmandtorsion} below, which improves slightly \cite[Main Theorem]{cx} in the case where the abelian variety has complex multiplication.

\begin{proof}[Proof of Proposition \ref{propositioncmandtorsion}]
If $A/L$ does not have good reduction, then by Proposition \ref{torsiononcmbadreduction} we have that $$|A(L)_{tors}| \leq |\mu(K)| \cdot p^{2g \gamma_p(e|\mu(K)|)}.$$ On the other hand, if $A/K$ has good reduction then, since the toric and unipotent ranks are zero i.e., $\mu=\alpha=0$ in their notation, while the abelian rank is $g$, i.e., $\beta=g$ in their notation, \cite[Part (ii) Main Theorem]{cx} implies that  $$|A(L)_{tors}| \leq  \lfloor (1+\sqrt{q})^2 \rfloor^g \cdot p^{ 2g \gamma_p(e)} \}.$$
\end{proof}

\begin{remark}
    In the literature there exist global bounds on the torsion of CM abelian varieties over number fields (see \cite{gaudronremond} and \cite{silverberg1988}), which rely on the main theorem of complex multiplication and class field theory. On the other hand, our result is local, i.e., it only depends on local invariants, it is much more elementary, and it is of most interest when $A/L$ has bad reduction while $|\mu(K)|$ and $e$ are small.
\end{remark}

\bibliographystyle{plain}
\bibliography{bibliography.bib}

\end{document}